\documentclass[11pt]{amsart}

\usepackage{amsmath,graphicx,url}

\usepackage[mathlines]{lineno}
\usepackage[colorlinks=false]{hyperref}

\usepackage{cleveref}

\theoremstyle{plain}
\newtheorem{theorem}{Theorem}[section]
\newtheorem{lemma}[theorem]{Lemma}
\newtheorem{proposition}[theorem]{Proposition}
\newtheorem{claim}{}[theorem]

\crefname{claim}{\hspace{-0.35em}}{\hspace{-0.35em}}
\Crefname{claim}{\hspace{-0.35em}}{\hspace{-0.35em}}

\newcommand{\dash}{\nobreakdash-\hspace{0pt}}
\newcommand{\ba}{\backslash}

\title[Excluded minors for gammoids]
{The antichain of excluded minors for the class of gammoids is maximal}

\author[Mayhew]{Dillon Mayhew}
\address{School of Mathematics and Statistics,
Victoria University of Wellington,
New Zealand}
\email{dillon.mayhew@vuw.ac.nz}

\date{\today}

\begin{document}

\begin{abstract}
Every gammoid is a minor of an excluded minor for the class
of gammoids.
\end{abstract}

\maketitle

\addtocounter{section}{1}

In \cite{Gee08}, Geelen conjectured that for every real-representable matroid,
$N$, there is an excluded minor, $M$, for the class of real-representable
matroids, such that $M$ has an $N$\dash minor.
Mayhew, Newman, and Whittle proved this in \cite{MNW09},
and showed that the statement holds even when we replace the real
numbers with an arbitrary infinite field.
This is distinctly different to the situation for finite fields,
where we know the number of
excluded minors to be finite, thanks to the resolution of
Rota's conjecture \cite{GGW14}.

In a comment on the \url{matroidunion.org} blog, Geelen
asked whether the same property holds for the class of gammoids \cite{VZ13}.
It is this question that we answer here.

\begin{theorem}
\label{maintheorem}
Let $N$ be a gammoid.
There is an excluded minor, $M$, for the class of gammoids, such that
$M$ has an $N$\dash minor.
\end{theorem}

A sketch of the proof appears as a comment on the same blog post
where Geelen posed the question.
This article corrects an error in the construction given there, while
fleshing out some additional details.
The broad strategy of the proof is similar to that in \cite{MNW09}.

Note that \Cref{maintheorem} implies that every matroid
either has an excluded minor for gammoids as a minor,
or is contained as a minor in such an excluded minor.
Equivalently, the excluded minors for the class of gammoids
form a maximal antichain in the minor ordering of matroids.
It also follows that there are infinitely many excluded minors
for the class of gammoids, a fact first noted by Ingleton \cite{Ing77}.

Gammoids were developed by Perfect \cite{Per68}, Mason \cite{Mas72},
and Ingleton and Piff \cite{IP73}.
We recall the definition here.
Let $G$ be a directed graph on the vertex set $V$.
Fix a subset $T\subseteq V$, and let $X$ be an arbitrary
subset of $V$.
If there is a collection of $|X|$ vertex-disjoint directed
paths such that each path begins with a vertex in $X$ and
ends with a vertex in $T$, then we shall say that
$X$ is \emph{linked} to $T$, and we shall call the collection of
paths a \emph{linking}.
If $x$ and $y$ are the first and last vertices in a path belonging to the
linking, we say $x$ is \emph{linked} to $y$.
Note that a path in the linking may consist of a single vertex, in the case
that a vertex belongs to both $X$ and $T$.
The subsets of $V$ that are linked to $T$ form the independent
sets of a matroid on the ground set $V$ \cite[Theorem~2.4.4]{Oxl11}.
The restriction of this matroid to the subset $S\subseteq V$
is denoted $L(G,S,T)$.
Note that the independent sets of $L(G,S,T)$
are exactly the subsets of $S$ that are linked to $T$ in $G$.
Any matroid of the form $L(G,S,T)$ is a \emph{gammoid}.
If $S$ is the entire vertex set, $V$, then $L(G,S,T)$ is a \emph{strict gammoid};
hence every gammoid is a restriction of a strict gammoid.
A matroid is a strict gammoid if and only if it is a dual of
a transversal matroid \cite[Corollary~2.4.5]{Oxl11}.
Since transversal matroids are representable over all
infinite fields \cite{PW70}, the next result follows.

\begin{proposition}
\label{afghan}
Let $N$ be a gammoid.
Then $N$ is representable over every infinite field.
\end{proposition}

The next result is \cite[Proposition~3.2.12]{Oxl11}.

\begin{proposition}
\label{whisky}
The class of gammoids is closed under minors.
\end{proposition}

\begin{lemma}
\label{infant}
Let $G$ be a directed graph, and let $S$ and $T$ be subsets of
vertices.
Let $N$ be the gammoid $L(G,S,T)$.
If $t$ is a vertex in $S\cap T$, then $N/t=L(G-t,S-t,T-t)$.
\end{lemma}

\begin{proof}
Let $X$ be a subset of $S-t$.
If $X$ is independent in $N/t$, then $X\cup t$ is independent in $N$,
so $X\cup t$ is linked to $T$ in $G$.
The paths linking $X$ to $T$ cannot use $t$, so $X$ is linked
to $T-t$ in $G-t$.
Hence $X$ is independent in $L(G-t,S-t,T-t)$.
Now let $X$ be independent in $L(G-t,S-t,T-t)$.
Then $X$ is linked to $T-t$ in $G-t$, and hence in $G$.
But the paths linking $X$ to $T-t$ do not use the vertex $t$,
so $X\cup t$ is linked to $T$.
Thus $X\cup t$ is independent in $N$, and $X$ is independent in $N/t$.
\end{proof}

\begin{lemma}
\label{toupee}
Let $N$ be a gammoid on the ground set $S\cup T$, where $S\cap T=\emptyset$,
and $T$ is a basis of $N$.
There exists a directed graph, $G$, with vertex set $V$, such that
$S\cup T \subseteq V$, and $N=L(G,S\cup T, T)$.
\end{lemma}

\begin{proof}
Since $N$ is a gammoid, it is a restriction of a strict gammoid, $M$.
Let $T'$ be a basis of $M$ that contains $T$.
As $M$ is the dual of a transversal matroid, Lemma~2.4.4 in \cite{Oxl11}
implies there is a directed graph, $G'$, on vertex set $V'$,
such that $M=L(G',V',T')$.
Construct $G$ by deleting the vertices in $T'-T$ from $G'$.
\Cref{infant} and induction imply that
$L(G,S\cup T, T)$ is $(M/(T'-T))|(S\cup T)$.
This last matroid is equal to $M|(S\cup T)=N$, so we
are done.
\end{proof}

\begin{lemma}
\label{tipple}
Let $N$ be a matroid and assume that the element $x$ is freely placed
in $N$.
Then $N$ is a gammoid if and only if $N\ba x$ is.
\end{lemma}

\begin{proof}
Let $S$ be the ground set of $N$.
\Cref{whisky} implies that $N\ba x$ is a gammoid when $N$ is.
Assume $N\ba x$ is a gammoid.
If $x$ is a coloop in $N$, then it is easy to see that $N$ is a gammoid,
so assume $x$ is in a circuit.
Let $T$ be a basis of $N\ba x$.
By \Cref{toupee} we assume $N\ba x=L(G,S-x,T)$, for some directed graph $G$.
Construct $G'$ be adding the vertex $x$ to $G$ and adding
arcs from $x$ to every vertex in $T$.
Let $N'=L(G',S,T)$.
We claim $N=N'$.
Certainly $N\ba x=N'\ba x$.
Let $C$ be a circuit of $N$ that contains $x$.
Then $C$ is spanning, since $x$ is freely placed.
Hence $C-x$ is a basis in $N\ba x=N'\ba x$, so $C$ is
dependent in $N'$.
Since $C-x$ can be linked to $T$ in $G$, it can also be linked
to $T$ in $G'$, so $C-x$ is independent in $N'$.
If $y$ is an arbitrary element of $C-x$, then $C-\{x,y\}$
can be linked to $T$ in $G$, and hence in $G'$, and,
since $|C-\{x,y\}|=|T|-1$, we can also link $C-y=(C-\{x,y\})\cup x$ to
$T$ in $G'$.
Now it follows that $C$ is a circuit in $N'$.
Next assume $C$ is a circuit of $N'$ that contains $x$.
Then $C-x$ can be linked to $T$ in $G'$.
If $|C-x|<|T|$, then $C$ can be linked to $T$ in $G'$, a contradiction.
Therefore $|C-x|=|T|$, and $C$ is a spanning circuit.
Because $C-x$ is a basis of $N'\ba x=N\ba x$, it follows that $C$ contains
a circuit in $N$ that contains $x$.
As $x$ is freely placed in $N$, we deduce that $C$ is a circuit in $N$,
and the proof is complete.
\end{proof}

\begin{lemma}
\label{reaper}
Let $N$ be a gammoid.
Then $N$ is a minor of a gammoid, $N'$, such that
the ground set of $N'$ can be partitioned into two bases.
\end{lemma}

\begin{proof}
Assume the ground set of $N$ is $S\cup T$, where $S\cap T=\emptyset$,
and $T$ is a basis of $N$.
Let $G$ be a directed graph such that $N=L(G,S\cup T,T)$.
Let $X$ be a maximum-sized independent subset
of $S$, and let
$S-X$ be $\{v_{1},\ldots, v_{m}\}$.
Note $|X|\leq |T|$ and let $n=|T|-|X|$.
We construct the directed graph $G'$ as follows.
For each vertex $v_{i}$, add a vertex $t_{i}$ and an edge
directed from $v_{i}$ to $t_{i}$.
Add vertices $u_{1},\ldots, u_{n}$, and add edges directed
from each $u_{i}$ to all vertices in $T$.
We let $N'$ be
\[L(G',S\cup T\cup \{t_{1},\ldots, t_{m}\}\cup \{u_{1},\ldots, u_{n}\},
T\cup \{t_{1},\ldots, t_{m}\}).\]
Since each $v_{i}$ is linked to $t_{i}$, it follows that
$T\cup \{v_{1},\ldots, v_{m}\}$ is a basis of $N'$.
Note that $X\cup\{t_{1},\ldots, t_{m}\}\cup \{u_{1},\ldots, u_{n}\}$
has cardinality
\[|X|+(|S|-|X|)+(|T|-|X|)=|T|+|S|-|X|=|T\cup\{t_{1},\ldots, t_{m}\}|.\]
Moreover, $X$ is linked to $T$ in $G$, and by construction, so is
$X\cup\{u_{1},\ldots, u_{n}\}$.
It follows that $X\cup\{t_{1},\ldots, t_{m}\}\cup \{u_{1},\ldots, u_{n}\}$
is also a basis in $N'$, so the ground set of $N'$ is partitioned into
two bases.
It is clear from \Cref{infant} that
$N$ can be recovered from $N'$ by deleting
$\{u_{1},\ldots, u_{n}\}$, and contracting
$\{t_{1},\ldots, t_{m}\}$.
\end{proof}

\begin{proof}[Proof of \textup{\Cref{maintheorem}}]
\setcounter{theorem}{1}
Let $N$ be a gammoid.
By \Cref{reaper} we can assume that
the ground set of $N$ is $S_{1}\cup S_{2}$, where
$S_{1}\cap S_{2}=\emptyset$ and both $S_{1}$ and $S_{2}$
are bases of $N$.
Let $r$ be $|S_{1}|=|S_{2}|$.
It follows from \Cref{toupee} that there are directed graphs,
$A_{1}$ and $A_{2}$, such that
$N=L(A_{1},S_{1}\cup S_{2},S_{1})=L(A_{2},S_{1}\cup S_{2}, S_{2})$.
For each $i=1,2$, we will now construct further directed graphs,
$B_{i}$, $C_{i}$, and $D_{i}$.

Henceforth we let $\{i,j\}$ be $\{1,2\}$.
We construct the directed graph $B_{i}$ from $A_{i}$ as follows.
Assume that $S_{i}$ is $\{x_{1},\ldots, x_{r}\}$.
Relabel each vertex, $x_{k}$, as $x_{k}'$, and then
add new vertices $x_{1},\ldots, x_{r}$.
Add an arc from each vertex $x_{k}$ to $x_{k}'$, and let
$S_{i}$ be the set of new vertices, $\{x_{1},\ldots, x_{r}\}$.
Let $T_{i}$ be the set $\{x_{1}',\ldots, x_{r}'\}$.
Add new vertices $v_{i}$ and $v_{j}$.
Attach arcs directed from each vertex in $S_{i}$ to $v_{i}$, and
arcs directed from each vertex in $S_{j}$ to $v_{j}$.

\begin{claim}
\label{decade}
For $i\in \{1,2\}$, let $N_{i}$ be
$L(B_{i},S_{1}\cup S_{2}\cup\{v_{1},v_{2}\},T_{i}\cup \{v_{1},v_{2}\})$.
Then $N_{1}=N_{2}$.
\end{claim}

\begin{proof}
Note that $S_{i}\cup \{v_{1},v_{2}\}$ is a basis of $N_{i}$.
Now let $X$ be an arbitrary basis in $N_{i}$, so that $X$
is linked to $T_{i}\cup\{v_{1},v_{2}\}$ in $B_{i}$.
We first consider the case that $\{v_{1},v_{2}\}\cap X=\emptyset$.
There must be a vertex in $X$ that is linked to $v_{i}$.
The only vertices with arcs directed towards $v_{i}$ are
in $S_{i}$, and these vertices have indegree zero.
Therefore there is a vertex, $x$, in $X\cap S_{i}$ that is linked to $v_{i}$.
Each vertex, $x_{k}$, in $(X-x)\cap S_{i}$ must be linked to the corresponding
vertex, $x_{k}'$, in $T_{i}$.
Let $T_{X}$ be $\{x_{k}'\colon x_{k}\in (X-x)\cap S_{i}\}$.
There is a vertex in $X$ linked to $v_{j}$, and by the previous
arguments, it cannot belong to $S_{i}$.
Hence there exists a vertex, $y\in X\cap S_{j}$, that is linked to $v_{j}$.
Since $|(X-x)\cap S_{i}|+|(X-y)\cap S_{j}|=|X-\{x,y\}|=|T_{i}|$, we see that
$|(X-y)\cap S_{j}|=|T_{i}-T_{X}|$, and there is a linking from
$(X-y)\cap S_{j}$ to $T_{i}-T_{X}$.
It follows that $X-\{x,y\}$ can be linked to $S_{i}$ in $A_{i}$, and hence
$X-\{x,y\}$ is a basis of $N$.
Therefore $(X-x)\cap S_{i}$ can be linked to
$S_{j}-(X-y)$ in the graph $A_{j}$.
We use this linking to construct a linking from $X$ to
$T_{j}\cup \{v_{1},v_{2}\}$ in $B_{j}$,
and thereby show that $X$ is a basis in $N_{j}$ also.
Each vertex in $(X-x)\cap S_{i}$ can be linked using
the same path as in $A_{j}$ (relabelling each end vertex, $x_{k}$, as $x_{k}'$).
Any vertex in $(X-y)\cap S_{j}$ is linked to the corresponding vertex in $T_{j}$.
Finally, $x$ is linked to $v_{i}$, and $y$ is linked to $v_{j}$.

Almost identical arguments apply when $|\{v_{1},v_{2}\}\cap X|\ne 0$.
\end{proof}

Now we let $N'$ be the matroid $N_{1}=N_{2}$, as defined
in the statement of~\Cref{decade}.

\begin{claim}
\label{gelato}
$N'/\{v_{1},v_{2}\}=N$.
\end{claim}

\begin{proof}
This follows with very little effort from \Cref{infant}.
\end{proof}

To construct $C_{i}$ from $B_{i}$, we add three new vertices,
$w_{i}$, $c_{i}$, and $d_{i}$.
Draw arcs from $c_{i}$ to $w_{i}$ and $v_{j}$.
Draw arcs from $d_{i}$ to $w_{i}$ and $v_{i}$.
Now add $C$, a set of two new vertices, and arcs from the vertices in
$C$ to $c_{i}$ and to all vertices in $T_{i}$.
Finally, add $D$, a set of $r+1$ new vertices, and arcs from the vertices in
$D$ to $d_{i}$ and all the vertices in $T_{i}$.
\Cref{fig} shows a schematic representation of $C_{i}$.

\begin{figure}[htb]
\centering
\includegraphics{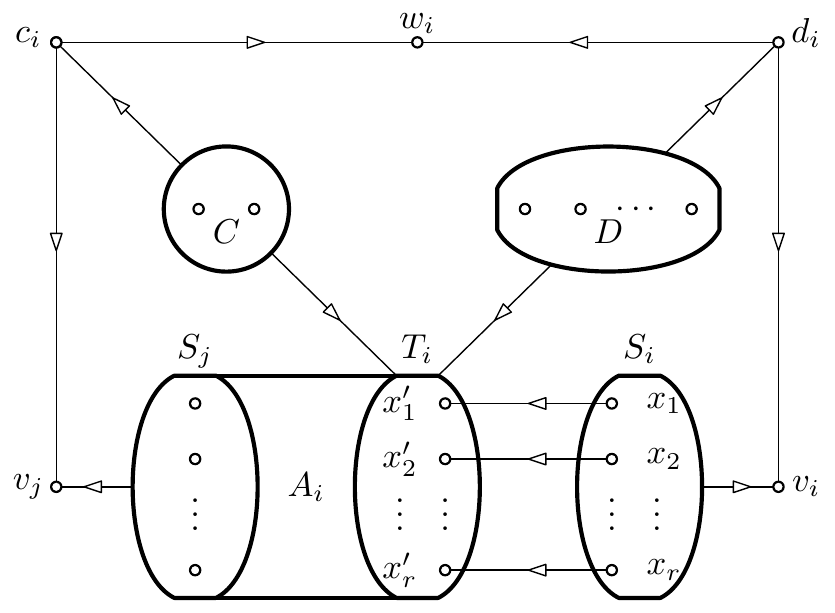}
\caption{The construction of $C_{i}$.}
\label{fig}
\end{figure}

Let $M_{i}'$ be the gammoid
$L(C_{i},S_{1}\cup S_{2}\cup C\cup D\cup\{v_{1},v_{2}\},
T_{i}\cup \{v_{1},v_{2},w_{i}\})$.

\begin{claim}
\label{frypan}
Let $X$ be a subset of $E(M_{i}')$ such that
$X\cap (C\cup D)\ne \emptyset$.
Then $X$ is a non-spanning circuit of $M_{i}'$ if and only if
$|X|=r+3$ and $X$ is a subset of one of $C\cup D$, $S_{i}\cup C\cup v_{i}$,
$S_{i}\cup D\cup v_{i}$, $S_{j}\cup C\cup v_{j}$, or $S_{j}\cup D\cup v_{j}$.
\end{claim}

\begin{proof}
It is straightforward to check that all the sets described in the statement
are non-spanning circuits in $M_{i}'$.

Let $X$ be a non-spanning circuit of $M_{i}'$ such that
$X\cap (C\cup D)\ne \emptyset$.
Let $x$ be an element in $X\cap (C\cup D)$.
Since $X-x$ is independent in $M_{i}'$, it is linked to a subset,
$L$, of $T_{i}'\cup\{v_{1},v_{2},w_{i}\}$ in $C_{i}$.
As $X$ is non-spanning, $|L|<r+3$.
It must be the case that $T_{i}\subseteq L$, for otherwise
we can link $x$ to a vertex in $T_{i}-L$, and thus link
$X$ to $T_{i}\cup\{v_{1},v_{2},w_{i}\}$ in $C_{i}$, a contradiction
as $X$ is dependent in $M_{i}'$.

If $X\cap (S_{i}\cup S_{j})=\emptyset$, then $X$
is a subset of $C\cup D\cup\{v_{1},v_{2}\}$.
We can assume $X$ contains $v_{1}$ or $v_{2}$, since otherwise
$X\subseteq C\cup D$ and we are done, as $C\cup D$ is a circuit.
Assume that $v_{j}$ is in $X$.
Since $v_{j}$ has outdegree zero, it follows that $v_{j}$ is linked
to itself.
Similarly, if $v_{i}$ is in $X$, it is linked to itself.
But if $v_{i},v_{j}\in X$, then $w_{i}\notin L$, as $|L|< r+3$.
Hence $x$ can be linked to $w_{i}$ via $c_{i}$ or $d_{i}$,
showing that $X$ is independent.
Therefore $X\cap \{v_{i},v_{j}\}=\{v_{j}\}$.
Since $C\cup D$ is a circuit, we assume that
$|X\cap (C\cup D)| \leq r+2$.
It therefore follows easily that $X\cap (C\cup D)$ can be linked to
$T_{i}\cup \{v_{i},w_{i}\}$, and hence $X$ is independent.
We similarly reach a contradiction if we assume that $v_{i}$ is in $X$.
Therefore we must conclude that $X$ contains elements from
either $S_{i}$ or $S_{j}$.

Assume that $X$ contains a vertex $s\in S_{i}$.
If $v_{i}\notin L$, then $s$ is linked to a vertex
$t\in T_{i}$.
But now we can link $s$ to $v_{i}$, and link $x$ to $t$, and
thus link $X$ to $T_{i}\cup\{v_{1},v_{2},w_{i}\}$.
This shows that if $X$ contains an element of $S_{i}$, then $L$ contains
$v_{i}$.
A symmetric argument shows that if $X$ contains an element of $S_{j}$, then
$L$ contains $v_{j}$.

Assume that $v_{i},v_{j}\in L$.
This means that $L=T_{i}\cup\{v_{i},v_{j}\}$, since $|L|<r+3$.
First assume that $x$ is in $C$.
Since $X$ is dependent, $x$ cannot be linked to $w_{i}$ without
using one of the vertices used to link $X-x$ to $L$.
This means that there is a vertex, $y\in C-x$, such that the path
that starts with $y$ uses $c_{i}$.
It follows that $y$ is linked to $v_{j}$.
Since $y$ is linked to $v_{j}$,
we deduce that $v_{j}$ is not in $X$.
If there is an element $s$ in $X\cap S_{j}$, then $s$ is linked to
a vertex $t\in T_{i}$.
We link $y$ to $w_{i}$ (via $c_{i}$), link $s$ to $v_{j}$,
and link $x$ to $t$.
This leads to a contradiction, so $X\cap (S\cup v_{j})=\emptyset$.
Now assume that there is an element $X\cap D$.
If no vertex in $X\cap D$ is linked to $v_{i}$, then
some $z\in X\cap D$ is linked to a vertex $t\in T_{i}$.
We can link $z$ to $w_{i}$ (via $d_{i}$), and then link $x$ to $t$,
demonstrating that $X$ is independent.
Hence there is a vertex $z\in X\cap D$ that is linked to $v_{i}$
via $d_{i}$, implying $v_{i}\notin X$.
If $s$ is an element in $X\cap S_{i}$, then $s$ is linked to $t\in T_{i}$,
and we can reroute by linking $z$ to $w_{i}$ (via $d_{i}$), linking $s$ to $v_{i}$,
and linking $x$ to $t$.
Therefore there is no element in $X\cap (S_{i}\cup v_{i})$, so
$X\subseteq C\cup D$.
Since $C\cup D$ is a circuit, it follows that $X=C\cup D$.
Therefore we assume there is no element in $X\cap D$, so
$X\subseteq C\cup S_{i}\cup v_{i}$.
As $|L|=r+2$, it follows that $|X|=r+3$, and we are done.
A symmetric argument shows that if $x$ is in $D$, then there is an
element $y\in D-x$ linked to $v_{i}$ via $d_{i}$, and either
$X=C\cup D$, or $|X|=r+3$ and $X\subseteq S_{i}\cup D\cup v_{i}$.
Therefore we can now assume $L$ does not contain both
$v_{i}$ and $v_{j}$.

Assume that $L$ does not contain $w_{i}$.
By earlier paragraphs, exactly one of the following holds:
$X$ contains an element in $S_{i}$ and $L=T_{i}\cup v_{i}$, or
$X$ contains an element in $S_{j}$ and $L=T_{i}\cup v_{j}$.
We will assume that the former case holds.
If $x$ is in $C$, then $x$ can be linked to $w_{i}$ via $c_{i}$.
This contradiction shows that $x$ is in $D$.
If there is no element in $(X-x)\cap D$ that is linked to $v_{i}$,
then we can link $x$ to $w_{i}$ via $d_{i}$.
Therefore some vertex $y\in (X-x)\cap D$ is linked to
$v_{i}$ via $d_{i}$.
Let $s$ be a vertex in $X\cap S_{i}$.
Then $s$ is linked to a vertex $t\in T_{i}$.
We reroute by linking $y$ to $w_{i}$ via $d_{i}$,
linking $s$ to $v_{i}$, and linking $x$ to $t$.
We reach a symmetric contradiction if we assume that
$L=T_{i}\cup v_{j}$.
Therefore we can now assume, using earlier
paragraphs, that $L=T_{i}\cup \{v_{i},w_{i}\}$
or $L=T_{i}\cup \{v_{j},w_{i}\}$.
We will assume that the former case holds.

By an earlier paragraph, we see that $X$ contains no vertex in $S_{j}$.
If $x$ is in $C$, then $x$ can be linked to $v_{j}$, which leads
to a contradiction.
Therefore $x$ is in $D$.
Assume that $C\cap X\ne \emptyset$.
If there is no vertex in $C\cap X$ that is linked to $w_{i}$,
then we choose a vertex $z\in C\cap X$, and we let
$t$ be the vertex in $T_{i}$ that $z$ is linked to.
We reroute by linking $z$ to $v_{j}$ via $c_{i}$, and linking
$x$ to $t$.
Therefore we can assume that $z\in C\cap X$ is linked
to $w_{i}$ via $c_{i}$.
If there is no vertex in $D\cap (X-x)$ that is linked to $v_{i}$,
then we can link $z$ to $v_{j}$ via $c_{i}$, and link $x$ to $w_{i}$
via $d_{i}$.
Thus $y\in D\cap (X-x)$ is linked to $v_{i}$ via $d_{i}$.
Now choose a vertex $s\in S_{i}\cap X$, and let $t$
be the vertex in $T_{i}$ that $s$ is linked to.
We link $z$ to $v_{j}$ via $c_{i}$,
link $y$ to $w_{i}$ via $d_{i}$,
link $s$ to $v_{i}$, and link $x$ to $t$.
This is impossible, so $X\cap C$ is empty.
Now $X$ is a subset of $S_{i}\cup D\cup v_{i}$,
and $|X|=|L|+1=r+3$, so we are done.
We reach a symmetric conclusion if we assume that
$L$ is $T_{i}\cup \{v_{j},w_{i}\}$.
\end{proof}

\begin{claim}
\label{blouse}
$M_{1}'=M_{2}'$
\end{claim}

\begin{proof}
Clearly $r(M_{1}')=r(M_{2}')=r+3$.
We consider the restriction $M_{i}'|(S_{i}\cup S_{j}\cup\{v_{i},v_{j}\})$.
Since vertices in $C\cup D$ have indegree zero in $C_{i}$,
we may as well delete these vertices.
Now $c_{i}$ and $d_{i}$ have indegree zero, so we delete these vertices
as well, and then delete $w_{i}$, since it is now isolated.
This argument shows that
$M_{i}'|(S_{i}\cup S_{j}\cup\{v_{i},v_{j}\})=N'$.
Let $X$ be a non-spanning circuit of $M_{i}'$.
If $X\cap (C\cup D)=\emptyset$, then $X$ is a circuit
of
\[M_{i}'|(S_{i}\cup S_{j}\cup\{v_{i},v_{j}\})=N'
=M_{j}'|(S_{i}\cup S_{j}\cup\{v_{i},v_{j}\}),\]
so $X$ is also a circuit of $M_{j}'$.
On the other hand, \Cref{frypan} shows that $M_{1}'$
and $M_{2}'$ agree in the non-spanning circuits
that intersect $C\cup D$.
The claim follows.
\end{proof}

Let $M'$ be the matroid $M_{1}'=M_{2}'$.
We obtain the directed graph $D_{i}$ from $C_{i}$
by adding arcs from the vertices in $C$ to $v_{j}$.
Let $M_{i}''$ be the gammoid
$L(D_{i},S_{i}\cup S_{j}\cup C\cup D\cup \{v_{1},v_{2}\},
T_{i}\cup \{v_{1},v_{2},w_{i}\})$.

\begin{claim}
\label{cutter}
Let $X$ be a subset of $E(M_{i}'')$ such that
$X\cap (C\cup D)\ne \emptyset$.
Then $X$ is a non-spanning circuit of $M_{i}''$ if and only if
$|X|=r+3$ and $X$ is a subset of one of
$S_{j}\cup C\cup v_{j}$, $S_{i}\cup D\cup v_{i}$, or
$S_{j}\cup D\cup v_{j}$.
\end{claim}

\begin{proof}
It is easy to check that the sets in the statement of the
claim are non-spanning circuits of $M_{i}''$.
If $X$ is a non-spanning circuit in $M_{i}''$ that intersects
$C\cup D$, then $X$ is also dependent in $M_{i}'$,
as every arc in $M_{i}'$ is also an arc in $M_{i}''$.
Now we can deduce that $X$ contains one of the
sets in the statement of \Cref{frypan}.
However, it is easy to check that $C\cup D$
and $S_{i}\cup C\cup v_{i}$
are independent in $M_{i}''$.
The claim follows.
\end{proof}

The next claim is easy to verify.

\begin{claim}
\label{velcro}
$C\cup D$ is a circuit-hyperplane of $M'$.
\end{claim}

We let $M$ be the matroid produced from $M'$ by relaxing
the circuit-hyperplane $C\cup D$.
The rest of the proof involves demonstrating that $M$ is an
excluded minor for the class of gammoids and that $M$ has an
$N$\dash minor.

\begin{claim}
\label{skewer}
$M$ has an $N$\dash minor.
\end{claim}

\begin{proof}
This follows from \Cref{gelato} because
$M\ba (C\cup D)=M'\ba (C\cup D)$
\cite[Proposition~3.3.5]{Oxl11},
and both these matroids are equal to $N'$.
\end{proof}

\begin{claim}
\label{mohawk}
$M$ is not a gammoid.
\end{claim}

\begin{proof}
Note that $S_{i}\cup v_{i}$ and $S_{j}\cup v_{j}$ both have
rank $r+1$ in $M'$ and hence in $M$.
Also, $S_{i}\cup S_{j}\cup C\cup \{v_{i},v_{j}\}$ and
$S_{i}\cup S_{j}\cup D\cup \{v_{i},v_{j}\}$ are both spanning in $M$,
as is $C\cup D$, since it is a basis.
Define $A$ and $B$ to be $S_{i}\cup v_{i}$ and $S_{j}\cup v_{j}$,
respectively.
Now
\begin{linenomath}
\begin{multline*}
r_{M}(A)+r_{M}(B)+r_{M}(A\cup B\cup C)+r_{M}(A\cup B\cup D)+r_{M}(C\cup D)\\
=2(r+1)+2(r+3)+(r+3)=5r+11.
\end{multline*}
\end{linenomath}
On the other hand, we can easily confirm that
$A\cup B$, $A\cup C$, $A\cup D$, $B\cup C$, $B\cup D$ all
have rank $r+2$ in $M$.
Therefore
\[
r_{M}(A\cup B)+r_{M}(A\cup C)+r_{M}(A\cup D)+
r_{M}(B\cup C)+r_{M}(B\cup D)=5r+10.
\]
From this it follows that $M$ violates the
Ingleton inequality \cite{Ing71}, so $M$ is not
representable over any field.
\Cref{afghan} implies that $M$ is not a gammoid.
\end{proof}

\begin{claim}
\label{ledger}
If $x$ is in $C\cup D$, then $M\ba x$ and $M/x$ are both gammoids.
\end{claim}

\begin{proof}
Since $M\ba x=M'\ba x$, it follows that $M\ba x$ is a gammoid.
If $x$ is in $C$, let $y$ be an arbitrary element in $D$, and otherwise
let $y$ be an arbitrary element of $C$.
Then \Cref{frypan} shows that the only non-spanning circuits in
$M$ containing $y$ do not contain $x$, and any such circuit
spans a hyperplane of $M$.
Thus $y$ is in no non-spanning circuit of $M/x$; that is,
$y$ is freely placed in $M/x$.
But the first part of this argument implies that $M\ba y$, and hence
$M/x\ba y$, is a gammoid.
Now \Cref{tipple} implies that $M/x$ is a gammoid, as desired.
\end{proof}

To complete the proof, we now let $x$ be an element in
$S_{i}\cup S_{j}\cup \{v_{i},v_{j}\}$, and we
show that $M/x$ and $M\ba x$ are gammoids.
The first case is simple, since $M/x=M'/x$.
Thus we consider $M\ba x$.
We let $\{i,j\}=\{1,2\}$, where $x$ is in $S_{i}\cup v_{i}$.
The non-spanning circuits of $M\ba x$ are those that do not
intersect $C\cup D$ (that is, the non-spanning circuits of $N'\ba x$),
along with any $(r+3)$\dash element subset
of $S_{j}\cup C\cup s_{j}$, $S_{j}\cup D\cup s_{j}$, or
$((S_{i}\cup v_{i})-x)\cup D$.
But it follows easily from \Cref{cutter} that these are the
non-spanning circuits of $M_{i}''\ba x$, so
$M\ba x=M_{i}''\ba x$, and hence the proof of
\Cref{maintheorem} is complete.
\end{proof}

%\bibliographystyle{/Users/Home/LaTeX/Bibliography/MRStyle}
%\bibliography{/Users/Home/LaTeX/Bibliography/References}

%\bibliographystyle{/Users/Dillon/Maths/LaTeX/Bibliography/MRStyle}
%\bibliography{/Users/Dillon/Maths/LaTeX/Bibliography/References}

\end{document}